\documentclass[10pt]{amsart}
\usepackage[top=4.5cm,bottom=4.5cm,left=4.3cm,right=4.3cm]{geometry}               
\usepackage{setspace} 
\usepackage{amsmath,amscd,amssymb,amsthm}
\usepackage[english]{babel}
\usepackage{mathtools}
\usepackage{enumerate}
\usepackage[numbers]{natbib}
\usepackage[hidelinks]{hyperref}
\usepackage{tikz}
\usepackage{tkz-graph}

\DeclareMathOperator{\disc}{Disc}
\DeclareMathOperator{\res}{Res}
\DeclareMathOperator{\Char}{Char}
\DeclareMathOperator{\GL}{GL}
\DeclareMathOperator{\Div}{Div}
\DeclareMathOperator{\inv}{inv}
\DeclareMathOperator{\supp}{supp}

\newcommand{\vF}{\mathbb{F}}
\newcommand{\vN}{\mathbb{N}}
\newcommand{\vZ}{\mathbb{Z}}

\newcommand{\vQ}{\mathbb{Q}}
\newcommand{\vH}{\mathbb{H}}
\newcommand{\vD}{\mathbb{D}}
\newcommand{\vP}{\mathbb{P}}

\newcommand{\cO}{\mathcal{O}}
\newcommand{\cS}{\mathcal{S}}
\newcommand{\cD}{\mathcal{D}}
\newcommand{\cL}{\mathcal{L}}

\newcommand{\RN}[1]{%
  \textup{\uppercase\expandafter{\romannumeral#1}}%
}

\newtheorem{theorem}{Theorem}[section]
\newtheorem{lemma}[theorem]{Lemma}
\newtheorem{proposition}[theorem]{Proposition}
\newtheorem{corollary}[theorem]{Corollary}
\theoremstyle{definition}
\newtheorem{problem}[theorem]{Problem}
\newtheorem{definition}[theorem]{Definition}
\theoremstyle{remark}
\newtheorem{remark}[theorem]{Remark}

\numberwithin{equation}{section}
\begin{document}

\title[A Local to global principle for densities]
{A local to global principle for densities over function fields}

\author[Giacomo Micheli]{Giacomo Micheli}
\address{Mathematical Institute\\
University of Oxford\\
Andrew Wiles Building\\ 
Woodstock~Road \\Oxford, UK
}
\email{giacomo.micheli@ox.ac.uk}

\subjclass[2010]{}


\begin{abstract}
Let $d$ be a positive integer and $\vH$ be an integrally closed subring of a global function field $F$.
The purpose of this paper is to provide a general sieve method to compute densities of subsets of $\vH^d$ defined by  local conditions. The main advantage of the method relies on the fact that one can use results from  measure theory to extract density results over $\vH^{d}$.
Using this method we are able to compute the density of the set of polynomials with coefficients in $\vH$ which give rise to ``good'' totally ramified extensions of the global function field $F$. As another application, we give a closed expression for the density of rectangular unimodular matrices with coefficients in $\vH$ in terms of the $L$-polynomial of the function field.
\end{abstract}

\maketitle

\section{Introduction}
\label{sec:introduction}
In \cite[Lemma 20]{bib:poonenAnn} B. Poonen and M. Stoll formalise a nice sieve method for computing densities using $p$-adic analysis. Essentially, the method consists of writing a given set $U\subseteq \vZ^d$ in terms of local conditions at the completions of $\vQ$;
once this is done, the density of $U$ can be computed by determining the measures of certain sets $U_p\subseteq \vZ_p$ which are associated to the local conditions which define $U$.
It is worth mentioning that the result is a powerful evolution of Ekedhal's Sieve (See \cite{torsten1991infinite}).

In this paper we present the extension of this method to global function fields i.e. univariate function fields over finite fields. Let $\vH$ be a non-trivial integrally closed subring of a function field $F$ and $\cS$ be the set of places of $F$ where all the functions in $\vH$ are well defined. It is well known (see for example \cite[Theorem 3.2.6]{bib:stichtenoth2009algebraic}) that $\vH$ consists exactly of the intersection of all the valuation rings $\cO_P$ of $F$ for $P\in\cS$. Vice versa, it also holds that an arbitrary intersection of valuation rings of $F$ is an integrally closed subring \cite[Proposition 3.2.5]{bib:stichtenoth2009algebraic}).
We will be interested in computing the density of a subset $U$ of $\vH^d$.

Before doing so, we first need to specify what  we mean by ``density'' of $U$ in the function field context. Over the set of rational integers $\vZ$, the density of a subset $U\subseteq \vZ^d$ is computed by considering the sequence of ratios between the number of points of $U$ falling in the hypercube of side $2B$ and centred at the origin, and $(2B)^d$. If $\{a_B\}_{B\in \vN}$ is this sequence of ratios and $u$ is its limit (if exists), then we say that $U$ has density $u$.
In the case of $\vH^d$, we explain how to use Moore-Smith convergence 
\cite[Chapter 2]{bib:kelley1955general} to define a notion of limit over the directed set of positive divisors having support in the complement of  $\cS$ (see also \citep{bib:BS}). Once this is understood, Riemann-Roch spaces of positive divisors having support in the complement of $\cS$ will play the role of intervals, and therefore products of such spaces will play the role of hypercubes.
 
Let $D_\cS$ be the set of positive divisors having support in the complement of $\cS$. The striking analogy between $\vZ$ and $\vH$ which allows our density definition (see Subsection \ref{preliminarydefandnot}) is  given by
\[\vZ=\bigcup_{B\in \vN}[-B,B[\,\cap\, \vZ \quad \text{and} \quad \vH=\bigcup_{D\in \mathcal D_\cS} \cL(D).\] 
In particular the reader should notice that the definition of density we will provide is consistent with the one used in the literature in the case of $\vF_q[x]$: if $F=\vF_q(x)$, $\vH=\vF_q[x]$ and $P_\infty$ is the place at plus infinity with respect to $x$, we have that $D_\cS=\{n P_\infty\}_{n\in \vN}$ and therefore 
 \[\cL(D)=\cL(nP_\infty)=\{f\in \vF_q[x]: \; \deg(f)\leq n\},\]
which induces the natural definition of density in the context of $\vF_q[x]$.

The essence of the presented method (Theorem \ref{thm:main_density_function_field}) is to polarize the difficulty of the problem: in fact, on one hand the $p$-adic formalism allows to easily compute a ``candidate'' for the density of a certain subset of $\vH^d$ by using tools from measure theory, on the other hand all the difficulty of the problem is unloaded on proving that the limit of a certain sequence (given by Equation \eqref{fund_cond_loc_to_glob}) tends to zero. 
In particular, we show that whenever the local conditions are actually related in a certain way to polynomial equations, the limit can be proven to be always zero (Theorem \ref{condition_verified_polynomials_THEOREM}). 

The entire machinery we build in Section \ref{sec:loctoglobprinciplesection}  is then used to produce two new results in Sections \ref{sec:probabilityoftotram} and \ref{sec:unimodularmatrices}. 

In Section \ref{sec:probabilityoftotram} we compute the
 probability that a ``random'' polynomial $f$ of fixed degree with coefficient in an given integrally closed subring $\vH\subset F$ gives rise to a totally ramified extension $E=F[y]/(f(y))$ of $F$ for which the equation $f(y)=0$ is ``good enough'' around the totally ramified place (in terms of Definition
\ref{def:goodtotram}).

Let $k,m$ be positive integers such that $k<m$ and $R$ be a domain. The question whether a homomorphism of $R^k$ in $R^m$ can be extended to an automorphism of $R^m$ raised many interesting questions in the past (see for instance Serre's Conjecture, which is proven in \cite{suslin1974projective,bib:projmod}).
In Section \ref{sec:unimodularmatrices} we close the problem of computing the density of homomorphisms of $\vH^k$ in $\vH^m$ which can be extended to automorphims of $\vH^m$. In the case of $\vH=\mathbb F_q[x]$, these homomorphisms arise from context of convolutional codes (see for example \cite{FORNASINI2004119} or \cite{rosenthal2001connections}) and their density was studied in \cite{bib:guo2013probability} and \cite{micheli2016density}.
In Theorem \ref{unimatrixdensity} we show that the density of unimodular matrices over $\vH$ is a rational number and can be explicitly computed as soon as the complement of the holomorphy set $\cS$ is finite.

\subsection{Preliminary definitions and notations}\label{preliminarydefandnot}
Let $\vF_q$ be a finite field.
In this paper all the function fields are global and have full constant field $\vF_q$. We denote by $\cO_P$ a valuation ring of function field $F$, having maximal ideal $P$. The set of all the places of $F$ will be denoted by $\vP_F$.
If $\cS$ is a proper subset of $\vP_F$, we denote by $\cS_t$ the subset of places of $\cS$ of degree greater than $t$. 
Moreover, we write $\vH_\cS$ to denote the \emph{holomorphy ring of $\cS$} i.e. the intersection of all the valuation rings associated to the places of $\cS$:
\[\vH_\cS=\bigcap_{P\in \cS} \cO_P.\] 
Sometimes, we will refer to $\cS$ as the \emph{holomorphy set} of $\vH_\cS$ and to $\vH_\cS$ as the \emph{holomorphy ring} of $\cS$. Holomorphy rings are integrally closed in $F$ and any integrally closed subring of $F$ is an holomorphy ring \cite[Proposition 3.2.5, Theorem 3.2.6]{bib:stichtenoth2009algebraic}.
In the whole paper we consider only holomorphy rings whose holomorphy set has finite complement in the set of all places of $F$. The most immediate example of holomorphy ring is $\vF_q[x]$ as this consists of the intersection of all the valuation rings of $\vF_q(x)$ different from the valuation ring at infinity.

Let $\Div(F)$ be the set of divisors of $F$ i.e. the free abelian group having as base symbols the elements in the set $\vP_F$. 
For $D=\sum_{P\in \vP_F} n_P P\in \Div(F)$, we denote by $\supp(D)$ the finite subset of $\vP_F$ for which $n_P$ is non-zero. Moreover, we will write $D\geq 0$ whenever $n_P\geq 0$ for any $P$ in $\vP_F$. Let 
\[\Div^+(F)=\{D\in \Div(F)\;|\; D\geq 0\}\]
Let $\cD_\cS$ be the subset of divisors of $\Div^+(F)$ having support over the complement of $\cS$ in $\vP_F$.
As $\cD_\cS$ is a directed set, we can define via Moore-Smith Convergence (see \cite[Chapter 2]{bib:kelley1955general} and more specifically for this context \citep{bib:HolMS}) a notion of limit over $\cD_\cS$. In this context, we can give an upper density definition for a subset $A$ of $\vH_\cS^d$ as follows:
\[\overline \vD_{\cS}(A):=\limsup_{D\in\cD_{\cS}}\frac{|A\cap \cL(D)^d|}{q^{\ell(D)d}}\]
where $\cL(D)$ is the Riemann-Roch space attached to the divisor $D$ and $\ell(D)=\dim_{\vF_q}(\cL(D))$.
Analogously, one can give a notion of lower density $\underline \vD_\cS$ by considering the inferior limit of the sequence. Whenever these two quantities are equal, we say that a subset $A$ of $\vH_\cS^d$ has a well-defined density $\overline \vD_{\cS}(A)=\underline \vD_{\cS}(A)=\vD_{\cS}(A)$. 

For a valuation ring $\cO_P$ let us denote by $\widehat \cO_P$ the completion of $\cO_P$ with respect to the $P$-adic metric. In addition, let us denote by $\mu_P$ the normalized Haar measure on $\widehat{\cO}_P$ with respect to the $P$-adic metric. 
For a subset $U\subseteq \widehat \cO_P$ we denote by $\partial U$ the boundary of $U$ with respect to the topology induced by the $P$-adic metric.
For a multivariate polynomial $f\in F[x_1,\dots x_n]$, we will denote by $\deg_{x_i}(f)$ (resp. $\deg_{\hom}(f)$) the degree of $f$ in the variable $x_i$ (resp. the degree of the homogenization of $f$). Whenever $f$ has all the coefficients in a given valuation ring $\cO_P$, we will denote by $\deg^P_{x_i}(f)$ (resp. $\deg^P_{\hom}(f)$) the degree of $f$ in the variable $x_i$ (resp. the degree of the homogenization of $f$) in $(\cO_P/P)[x_1,\dots,x_n]$.
For a positive integer $n$ and a given commutative domain $R$, we will denote by $\GL_n(R)$ the set of $n\times n$ matrices whose determinant is a unit of $R$.  

\section{The local to global principle for densities over global function fields}\label{sec:loctoglobprinciplesection}
In this section we describe the local to global principle which will be used later on.
This result is the function field analogue of \citep[Lemma 20]{bib:poonenAnn}.
\begin{theorem}\label{thm:main_density_function_field}
Let $d$ be a positive integer, $\cS$ be a subset of places of $F$ and $\vH_{\cS}$ the holomorphy ring of $\cS$. 
For any $P\in \cS$, let $U_P\subseteq \widehat \cO_P^d$ be a measurable set such that $\mu_P(\partial U_P)=0$.
Suppose that
\begin{equation}\label{fund_cond_loc_to_glob}
\lim_{t\rightarrow \infty }\overline{\vD}_{\cS}(\{a\in \vH_{\cS}^d\:|\: a\in U_P \;\text{for some }\; P\in \cS_t\})=0.
\end{equation}
Let $\pi:\vH_{\cS}^d\longrightarrow 2^{\cS}$ defined by
$\pi(a)=\{P\in \cS: a\in U_P\}\in 2^{\cS}$. Then
\begin{itemize}
\item[(i)] $\displaystyle{\sum_{P\in \cS} \mu_P(U_P)}$ is 
convergent.
\item[(ii)] Let $\Gamma\subseteq 2^{\cS}$. Then $\nu (\Gamma):=\vD_{\cS}(\pi^{-1}(\Gamma))$ exists and  $\nu$ defines a measure on $2^{\cS}$.
\item[(iii)] $\nu$ is concentrated at finite subsets of $\cS$. In addition, if $T\subseteq \cS$ is finite we have:
\[\nu(\{T\})=\left(\prod_{P\in T}\mu_P(U_P)\right)\prod_{P\in \cS\setminus T } (1-\mu_P(U_P)).\]
\end{itemize}
\end{theorem}
\begin{proof}
Throughout the proof $\cS$ will be fixed, so we will denote $\vH_\cS$ by $\vH$.
What we need to do is to translate the proof of \cite{bib:loctoglob} to the context of function fields.  
Essentially, we need to understand how the measure of $P$-adic intervals can be translated into density via the use of Riemann-Roch Theorem 
\cite[Theorem 1.5.15]{bib:stichtenoth2009algebraic}. Once this is done, the same arguments of the proof of \citep[Lemma 20]{bib:poonenAnn} will apply to this context.
We define a \emph{$P$-interval} in $\widehat \cO_P$ as the set $\{x\in \widehat \cO_P: x\equiv a \mod P^{e_P}\}$ for some $e_P\in \vN$ and $a\in \widehat \cO_P$. 
A $P$-\emph{box} $I_P$ will just be a product of $P$-intervals:
\[I_P=\{x\in \widehat \cO_P^d: x_j\equiv a_j \mod P^{e_{P,j}} \; \text{for $j\in \{1,\dots d\}$} \}.\]
To simplify notation, we say that a $P$-box is a $P$-\emph{cube} if it has the form 
\[C_P=\{x\in \widehat \cO_P^d: x_i\equiv y_i \mod P^{e_P}, \; i\in \{1,\dots,d\}\}\] for some $e_P\in \vN$ and $(y_1,\dots,y_d)\in \widehat \cO_P^d$. In other words, $C_P$ is the cartesian product of intervals of equal length. We say that $c_P=(y_1,\dots,y_d)$ is the \emph{center} of the cube. 
Let $A$ be a finite subset of $\cS$.
Let us now compute the density of the elements in $\vH^d$ which are mapped in a product of a finite number of $P$-boxes via the natural embedding
$\vH^d\longrightarrow \prod_{P\in A} \widehat \cO_P^d$. 
Let $I=\prod_{P\in A}I_P$ be such product of $P$-boxes. For any $P$, the $P$-box $I_P$ can be covered with a finite number $l_P$ of disjoint cubes of equal size $e_P$, as all the congruences can be decomposed in terms of the finest congruence, given by $\max\{e_{P,j}:\; {j\in \{1,\dots d\}}\}=:e_P$.
Therefore, one can write
\[I=\prod_{P\in A}I_P=\prod_{P\in A} \bigsqcup^{l_P}_{i=1} C_P^{(i)}\]
with $\mu_P(C_P^{(i)})=q^{-d\deg(P)e_P}$ independently of $i$.

We consider the diagram
\[
\begin{CD}
\vH^d        @>\iota>>    \prod_{P\in A} \widehat \cO_P^d \supseteq I\\
@VV\pi_J V                            @VV\pi'_J V\\
(\vH/J)^d   @>{\psi}>>    \prod_{P\in A} (\widehat \cO_P/P^{e_P})^d=:R
\end{CD}
\]
where $J=\prod_{P\in A} P^{e_P}\subseteq \vH$, the map $\iota$ is the natural inclusion, and $\psi$ is the isomorphism (coming from the Chinese Remainder Theorem) which makes the diagram commutative.
On the right-hand side, we can immediately compute the product measure $m$ of $I$ by looking at
its definition, getting $m=\prod_{P\in A} l_P q^{-d\deg(P)e_P}$.
It remains to show that the density of $\iota^{-1}(I)$ is indeed $m$. 
For this, let us decompose $I$.
Let $\mathcal I_P$ be a finite set indexing the cubes which cover $I_P$ and let $\mathcal I =\prod_{P\in A} \mathcal I_P$. Any given $i=(i_P)_{P\in A}\in \mathcal I$ determines a choice of cubes as follows: for each place $P\in A$ we select exactly one cube $C^{(i_P)}_P$, having center $c_P^{(i_P)}$. We now build a $C_i$ as the product
$\prod_{P\in A} C_P^{(i_P)}$. Clearly, the set of  $C_i$'s built in this way has cardinality $\prod_{P\in A} l_P$ and covers $I$ via a disjoint union. If we can now prove that the density of $\iota^{-1}(C_i)$ is independent of the choice of $i\in\mathcal I$, then we will have that 
\begin{equation}\label{eq:cubesintervals}
\vD(\iota^{-1}(I))=\vD(\iota^{-1}(C_i))\cdot \prod_{P\in A} l_P.
\end{equation}
To achieve this, we now explicitly compute the value $\vD(\iota^{-1}(C_i))$.
As the diagram above  is commutative, we can equivalently compute the density of elements of $\vH^d$ falling into $\psi^{-1}\pi_J'(C_i)$ via the map $\pi_J$. Let $z_{C_i}\in \pi_J^{-1}\psi^{-1} \pi_J' ((c_P^{(i)})_{P\in A})$. 
Notice that we have
\[\iota^{-1}(C_{i})=\pi_J^{-1}\psi^{-1}\pi_J'(C_{i})=z_{C_i}+J\vH. \]
Observe that for any divisor $D\in \cD_\cS$, the map $\pi_J$ restricted to $\cL(D)$ is $\vF_q$-linear. 
Let $g$ be the genus of $F$.
Therefore if we denote by $z_{C_i,j}$ the $j$-th component of $z_{C_i}$, we have
\[|\cL(D)\cap (z_{C_i,j}+J\vH)|=|\cL(D)\cap J\vH|=\left|\cL\left(D-\sum_{P\in A} e_P P\right)\right|=q^{\ell(D-\sum_{P\in A} e_P P)},\] 
which for $D$ of large degree, equals $q^{\deg(D-\sum_{P\in A} e_P P)+1-g}$ by Riemann-Roch Theorem. We can finally compute the density of elements mapping in the cube $C_i$ (which is in fact independent of $i$, as we wanted):
\[\begin{split}
\vD (\iota^{-1}(C_i)) & =\lim_{D\in \cD_\cS} \frac{|\cL(D)^d\cap  \pi_J^{-1}\psi^{-1}\pi_J'(C_i)|}{q^{\ell(D)d}}\\
 & =q^{-\ell(D)d}\prod^d_{j=1}|\cL(D)\cap (z_{C_i,j}+J\vH)|=q^{-(\sum_{P\in \cS_t}e_P\deg(P))d}.
\end{split}\]
Using now Equation \eqref{eq:cubesintervals} we get the final claim by comparing $m$ with $\vD(i^{-1}(I))$. 
Since now we have proved the theorem for boxes, all the arguments of the proof of \cite[Lemma 20]{bib:poonenAnn} are now straightforward to apply. In fact, suppose for a moment that the set of $P$'s in $\cS$ for which $U_P$ is different from the empty set is a finite set $A$. Now, let $T$ be a finite set of places. Assuming that $\mu_P(\partial U_P)=0$,  one can cover each of the $U_P$'s from the interior (resp. $U_P^{c}$) with a finite set of boxes which well approximate the measure $\mu_P(U_P)$ (resp. $\mu_P(U_P^{c})$). In particular one has
\[\nu(\{T\}) \geq \left(\prod_{P\in T}\mu_P(U_P')\right)\prod_{P\in \cS\setminus T } (1-\mu_P(U_P')),\]
where the products above are both finite and  $U_P'$ union of the boxes for each $P$, where the theorem holds. As we can apply the symmetric argument with a set of external approximations $U_P''$ we have
\[\nu(\{T\}) \leq \left(\prod_{P\in T}\mu_P(U_P'')\right)\prod_{P\in \cS\setminus T } (1-\mu_P(U_P'')),\]
from which the claim follows by letting the approximation get sharper and then $\mu_P(U_P'')$ and $\mu_P(U_P')$ tend to $\mu(U_P)$.

On the other hand, if $A$ is an infinite set, one easily sees that an approximation with finitely many $U_P$ is good enough, as long as condition \eqref{fund_cond_loc_to_glob} is verified. To see this, let $T$ be a finite subset of $\cS$ and let us recall that $\cS_t$ is the set of places of $\cS$ of degree larger than $t$ and then $\cS_t^{c}$ is the subset of $\cS$ consisting of places of degree less than or equal to $t$. Observe that
for a positive integer $t$ such that $\cS_t^c$ contains $T$ we can define a partial approximation of $\pi^{-1}(\{T\})$
\[W_t=\{a\in \vH_{\cS}^d\mid a \in U_P \;\forall P\in T,\; a\notin U_P\; \forall P\in T^c \cap \cS_t^c \}. \]
Notice that $W_t$ contains $\pi^{-1}(\{T\})$ so $\vD(\pi^{-1}(\{T\}))\leq \vD(W_t)$. In addition we have that
\[\vD(\pi^{-1}(\{T\}))\geq \vD(W_t)- 
\overline \vD (W_t\setminus \pi^{-1}(\{T\})).\]
Now, by letting $t$ go to infinity and using condition \eqref{fund_cond_loc_to_glob} on
$\vD (W_t\setminus \pi^{-1}(\{T\}))$ one gets the claim.
\end{proof}
The next Theorem ensures that when the $U_P$ can be expressed in terms of polynomial equations, Condition \eqref{fund_cond_loc_to_glob} is always verified, similarly to what happens in the case of Ekedhal Sieve for integers \cite{torsten1991infinite}.
\begin{theorem}\label{condition_verified_polynomials_THEOREM}
Let $F/\vF_q$ be a global function field and $\cS$ be a subset of $\vP_F$ with finite complement. Let $\vH_{\cS}$ be the holomorphy ring of $\cS$. Let $f,g \in \vH_{\cS}[x_1,\dots,x_d]$ be coprime polynomials. Then
\begin{equation}\label{condition_verified_polynomials}
\lim_{t\rightarrow \infty} \overline{\vD}_{\cS}\left(\{y\in \vH_{\cS}^d: \quad f(y)\equiv g(y)\equiv 0 \mod P\quad  \text{for some}\; P\in \cS_t\}\right)=0.
\end{equation}
\end{theorem}
\begin{proof}
If $d=1$  there is nothing to prove so we can suppose $d>1$. Without loss of generality, we can also suppose $\deg_{x_1}(f)>0$.
Since $\cS$ will be fixed throughout the proof, we will denote $\vH_{\cS}$ and $\vD_{\cS}$ by $\vH$ and $\vD$ respectively.
Let us recall that the places in $\cS$ are in natural correspondence with the prime ideals of $\vH$, therefore with a small abuse of terminology we will identify this two sets. 
We first fix $t$ large enough, so that $\deg^P_{x_1}(f)=\deg_{x_1}(f)$ for any $P$ of degree larger than $t$.
Now fix $D$ large enough so that $\deg(D)>t$.
Let us also introduce new notation to simplify the computations. For a divisor $D$, let us define
\[a_t(D):=\left| \{y \in \cL(D)^d \;:\; f(y)\equiv g(y)\equiv 0 \mod P\quad  \text{for some}\; P\in \cS_t \}\right|q^{-d\ell(D)},\]
\[c_P(D):=\left|\{y \in \cL(D)^d\;:\; f(y)\equiv g(y)\equiv 0 \mod P\}\right|q^{-d\ell(D)}.\]
Our first purpose is to estimate $a_t(D)$ for $t$ and $D$ large. 
First, we notice a simple upper bound for $a_t(D)$:
\[a_t(D)\leq \sum_{P:\deg(P)>t} c_P(D).\]
We now want to estimate the sum above for different regimes of $\deg(P)$ and $\deg(D)$. In order to do so,
let us  further  split the sum as 
\begin{equation}\label{splitted_equation}
\underbrace{\sum_{P\, : \, t<\deg(P)\leq \deg(D)} c_P(D)}_{(\RN{1})}+ \underbrace{\sum_{P\, : \, \deg(P)>\deg(D)} c_P(D)}_{(\RN{2})}.
\end{equation}

Let us estimate (\RN{1}). First, we want to give a reasonable estimate for $c_P(D)$ in the specified regime.
Notice that for each point of $z\in \vF^d_{q^{\deg(P)}}$ satisfying $f(z)\equiv g(z)\equiv 0 \mod P$ there are at most
$|\cL(D-P)^d|$ preimages of $z$ in $\cL(D)^d$, as the evaluation map $\cL(D)\rightarrow \cL(D)(P)\subseteq \vF_{q^{\deg(P)}}$ is linear and has kernel $\cL(D-P)$. Let $N_P$ be the number of $\vF_{q^{\deg(P)}}$-points of the variety defined by $f$ and $g$ when reduced modulo $P$. 
Let $g_F$ be the genus of $F$. 
By observing that $\ell(D)\geq \deg(D) + 1 - g_F$ and that $\ell(D-P)\leq \deg(D)-\deg(P)+1$ we get:

\begin{align*}
c_P(D)\leq & N_P |\cL(D-P)|^d q^{-\ell(D)d}\\
 & \leq N_P q^{d(\deg(D)-\deg(P)+1)} q^{-(\deg(D)+1-g_F)d}\\
 & = N_Pq^{(g_F-\deg(P))d}.
\end{align*}

As $t$ can be chosen large enough to avoid the places of bad reduction, we can estimate classically $N_P$ as $C q^{(d-2)\deg(P)}$ for some constant $C$. It follows that
\[\sum_{P:t<\deg(P)\leq \deg(D)} c_P(D)\leq \sum_{P:t<\deg(P)\leq \deg(D)}  \widetilde{C} q^{-2\deg(P)}\]
for some other constant $\widetilde C$.

Let us estimate (\RN{2}). Let $(f,g)$ be the ideal generated by 
$f,g$ in $F[x_1,\dots,x_d]$ and let $J=(f,g)\cap F[x_2,\dots,x_d]$. Since $(f,g)$ has codimension $2$, $J$ is principal. Let $h\in F[x_2,\dots x_d]$ be the generator of $J$, which can be chosen with coefficients over $\vH$ by multiplying by an appropriate element in $\vH$. 
Let us also assume without loss of generality that $\deg_{x_2}(h)>0$.
Let now $D$ be so large that modulo every prime $P$ of degree larger than $\deg(D)$, we have $\deg_{hom}^P(h)=\deg_{hom}(h)$.
Consider now all the  elements of $\cL(D)^d$ ending with a fixed $r=(r_2,\dots,r_d)\in \cL(D)^{d-1}$ and for which $h(r)\neq 0$.  Let us estimate their contribution to each $c_P(D)$ in the sum (\RN{2}). 
Let $I_r$ be the product of all the prime ideals $P$ of $\vH$ such that $\deg(P)>\deg(D)$ and for which there exists $x\in \vH$ such that $f(x,r_2,\dots,r_d)\equiv g(x,r_2,\dots,r_d)\equiv 0 \mod P$ (this set is finite as $h(r)\neq 0$). If we denote by $u_r$ the number of distinct primes appearing in the factorization of $I_r$, the contribution of all the $d$-tuples ending with $r$ is bounded by $u_r \deg_{x_1}^P(f)=u_r\deg_{x_1}(f)$.
By the definition of $h$, it is clear that $h \vH[x_1,\dots,x_d] \subseteq (f,g)\cap \vH[x_1,\dots,x_d]$. If we denote by $(f,g)_{I_r}$ the projection of $(f,g)\cap \vH[x_1,\dots,x_d]$ in $(\vH/I_r)[x_1,\dots,x_d]$, we have that 
\[h \vH/I_r[x_1,\dots,x_d] \subseteq (f,g)_{I_r}.\]
Therefore this  in turn implies that $(r_2,\dots r_d)$  satisfies $h(r_2,\dots, r_d)\equiv 0 \mod I_r$.
Now, the key observation to get the final estimate for (\RN{2}) is the following: 
$\deg(D)$ was chosen large in such a way that the homogeneous degree of $h$
is constant modulo $P$  for any $P$ of degree larger than $\deg(D)$. 
Recall now that every prime ideal $P$ appearing in the factorization of $I_r$ has degree larger than $D$, therefore
\[u_r\deg(D)< \deg(I_r)\leq \deg(D)\deg_{hom}(h)+C\]
where the constant $C$ depends on the leading coefficients of $h$ (and independent of $D$), from which it follows that $u_r<\deg_{hom}(h)$, for $D$ large enough.

The reader should now notice that in (\RN{2}), an element in $\cL(D)^d$ ending with $r$ (i.e. of the form $(r_1,r)$), cannot contribute more than $1$ for each $c_P(D)$, as the evaluation map $\cL(D)\rightarrow \cL(D)(P)$ is an injection to $\vF_{q^{\deg(P)}}$.
It follows easily that the set of all the $d$-tuples ending with $r$ contribute at most $u_r\deg_{x_1}(f)$ to the whole sum (\RN{2}).
Now, using the observations above and recalling that we also have to take into account the size of the set $T$ of the $(d-1)$-tuples $r=(r_2,\dots, r_d)$ such that $h(r)=0$ for $r\in \vH^{d-1}$, we finally get
\[\sum_{P:\deg(P)>\deg(D)} c_P(D)\leq \deg_{x_i}(h)q^{\ell(D)(d-1)}q^{-\ell(D)d}+ \sum_{r\in \cL(D)^{d-1}\setminus T} q^{-\ell(D)d}u_r\deg_{x_1}(f)\]
\[\leq q^{-\ell(D)} (\deg_{hom}(h)\deg_{x_1}(f)-1).\]

At this point we observe that the estimates of (\RN{1}) and (\RN{2}) only hold for large values of $t$ and $D$, so it is now important to notice how the order of the limits in \eqref{condition_verified_polynomials} is actually taken into account: in order for estimate (\RN{1}) to hold, it is enough to choose $t$ so large that the primes of bad reduction are avoided. For estimate (\RN{2}), take $D$ so large that 
\begin{itemize}
\item its degree is larger than $t$,
\item the homogeneous degree of $h$ is constant with respect to all places of $\cS$ of degree larger than $\deg(D)$,
\item the degree of $f$ with respect to $x_1$ is constant for all places of $\cS$ of degree larger than $\deg(D)$.
\end{itemize}
We can now safely use the two estimates to complete the proof:
\begin{align*}
\lim_{t\rightarrow \infty} \lim_{D\rightarrow \infty} a_t(D)\leq  \lim_{t\rightarrow \infty} \lim_{D\rightarrow \infty}\sum_{P:\deg(P)>t} c_P(D)=&\\
\lim_{t\rightarrow \infty} \lim_{D\rightarrow \infty} \sum_{P:\deg(P)> \deg(D)}c_P(D)+\sum_{P:t<\deg(P)\leq \deg(D)} c_P(D)\leq &\\
\lim_{t\rightarrow \infty} \lim_{D\rightarrow \infty} q^{-\ell(D)} (\deg_{hom}(h)\deg_{x_1}(f)-1)+ \sum_{P:t<\deg(P)\leq \deg(D)} \widetilde{C}q^{-2\deg(P)}=\\
\lim_{t\rightarrow \infty} \sum_{P:t<\deg(P)} \widetilde{C}q^{-2\deg(P)}.
\end{align*}
This completes the proof, since the sum above is the tail of a subseries of the zeta function of $F$ evaluated at $2$, which is converging.
\end{proof}

The reader should notice that the counting technique using in the estimate of (II) is similar to the one used in the case of $\vZ$ in the main result of \cite{torsten1991infinite}.

\section{On the probability of a totally ramified extension of global function fields}\label{sec:probabilityoftotram}

In this section we are interested in obtaining the ``probability'' that a random extension of a given function field is totally ramified in a good way at some place. 
If $F/\vF_q$ is a function field with full constant field $\vF_q$ and $E$ is a finite extension of $F$, we recall that an extension of places $Q|P$ is said to be \emph{totally ramified} if the dimension of  $\cO_Q/Q$ as an $\cO_P/P$ vector space is equal to $1$.

The notion of ``good'' total ramification is encoded in the following
\begin{definition}\label{def:goodtotram}
Let $E:F$ be a separable totally ramified extension of function fields. We say that $E:F$ is \emph{nicely totally ramified with respect to $f(T)\in F[T]$} if 
\begin{itemize}
\item $f(X)$ is the defining polynomial of $E$, i.e. $E\cong F[y]/(f(y))$.
\item there exists a totally ramified extension of places $Q\mid P$ of $E:F$ such that $\cO_P[y]=\cO_Q$.
\end{itemize}
\end{definition}
The reader with some insight in algebraic geometry can read the condition of nice total ramification as a condition of ``non-singularity'' around the totally ramified place. 
%
%

Our purpose is now to compute the density of degree $n$ polynomials $f(X)$ of $F[X]$ for which $f(X)$ is irreducible and the extension $F[y]/(f(y)):F$ is nicely totally ramified with respect to $f(X)$.

\subsection{A characterization of polynomials defining nice totally ramified extensions}

First of all, we convert the nice total ramification property into properties of the coefficients of the  defining polynomials. 
Let us start by stating some well known facts on totally ramified extensions.
\begin{lemma}\label{propcombinationstich}
Let $Q|P$ be an extension of places of the function field extension $E:F$. Let $u$ be a uniformizer for $Q$, $E=F(u)$ and $f_u$ be the minimal polynomial for $u$ over $\cO_P$. 
Then
\begin{itemize}
\item if $Q|P$ is totally ramified, we have that $\cO_P[u]=\cO_Q$;
\item $f_u$ is an Eisenstein polynomial if and only if $Q|P$ is totally ramified.
\end{itemize}
\end{lemma}
\begin{proof}
To see this, simply combine \cite[Proposition 3.1.15]{bib:stichtenoth2009algebraic} and \cite[Proposition 3.5.12]{bib:stichtenoth2009algebraic}.
\end{proof}
\begin{proposition}\label{fund_thm_ramified}
Let $k$ be a perfect field and $f$ be an irreducible separable polynomial of degree $n\geq 2$ over a function field $F/k$. Let $E=F[y]/(f(y))$ and $Q|P$ be a totally ramified extension of places of $E:F$. 
Then we have
\begin{itemize}
\item[(i)] if $y$ is holomorphic at $Q$, then $\cO_P[y]=\cO_Q$ if and only if $f(T+u)$ is an Eisenstein polynomial for some $u$ in $\cO_P$.
\item[(ii)] if $y$ is not holomorphic at $Q$, then $\cO_P[1/y]=\cO_Q$ if and only if $T^nf(1/T)$ is Eisenstein.
\end{itemize}
\end{proposition}
\begin{proof}
Let us prove (i). Suppose $\cO_P[y]=\cO_Q$, then there exists $u\in \cO_P$ such that $y(Q)=u(Q)=u(P)$. Let $z=y-u$.
The minimal polynomial of $z$ is $f(T+u)$, which implies that if we can prove that $v_Q(z)=1$, it will follow that $f(T+u)$ is Eisenstein with respect to $P$.
It is easily seen that $\cO_P[y]=\cO_P[z]=\cO_Q$, therefore if $x$ is a uniformizer for $Q$, we can write
\[x=\sum^{n-1}_{i=0}a_i z^i\]
for some $a_i\in \cO_P$. Whence,
\[1=v_Q\left(\sum^{n-1}_{i=0}a_i z^i\right)\geq \min_{i\in \{0,\dots,n-1\}}\{v_Q(a_i)+iv_Q(z)\}.\]
Since $x(Q)=z(Q)=0$, we get $v_Q(a_0)=nv_Q(a_0)\geq n>1$, which forces $v_Q(z)=1$, as $n\geq 2$.

Now suppose that $f(T+u)$ is Eisenstein with respect to some $P$. It follows that $z=y-u$ is a uniformizer for $Q$, which implies $\cO_P[y]=\cO_P[z]=\cO_Q$.

The claim (ii) easily follows by applying (i) to $\overline y=1/y$ with $u=0$.
\end{proof}
The following corollary shows that the polynomials which are Eisenstein after a Moebius transformation are not more than the ones which are Eisenstein after either a shift or an inversion. This result is not needed in the proof of the main theorem, nevertheless we include it for completeness (in the case of $\vZ$, this question was asked in \cite{bib:heyman2014shifted}).
\begin{corollary}
Let $f$ be a polynomial of degree $n\geq 2$ and $P$ a place of $F$. The following are equivalent:
\begin{itemize}
\item[(i)] There exist $h,s,l,j\in \cO_P$ such that $hj-sl\in \cO_P^*$ and $g(T)=(lT+j)^n f(\frac{hT+s}{lT+j})$ is Eisenstein;
\item[(ii)] One of the following two occurs: 
\begin{itemize}
\item $f(T+u)$ is Eisenstein with respect to $P$ for some $u\in \cO_P$.
\item $T^nf(1/T)$ is Eisenstein with respect to $P$.
\end{itemize}
\end{itemize}  
\end{corollary}
\begin{proof}
Clearly (ii) implies (i). Let us now prove that (ii) follows from (i). Let $m'(T)=\frac{aT+b}{cT+d}$ be the inverse of $m(T)=\frac{hT+s}{lT+j}$.
Suppose $y$ is a zero of $f(T)$ and is holomorphic at the totally ramified place $Q\subseteq F(y)$ lying over $P$. The element 
$\tilde y:=\frac{ay+b}{cy+d}$ is then a zero of the Eisenstein polynomial $g(T)$ therefore by 
\ref{propcombinationstich} and the fact that $F(y)=F(\tilde y)$ we have that $\cO_P[\tilde y]=\cO_Q$. 
Suppose $y$ is holomorphic at $P$, then it is enough to show that 
$\cO_P[y]=\cO_P[m(\tilde y)]=\cO_Q$.
To this end, we observe that
\[y(Q)=\frac{h(Q)\tilde y(Q)+s(Q)}{l(Q)\tilde y(Q)+j(Q)}=
\frac{s(Q)}{j(Q)}.\]
By the above equation it follows that the element 
\begin{equation}\label{equatonforz}
z=y-s/j=\frac{hj-sl}{j(l\tilde{y}+s)}\tilde y
\end{equation}
 has valuation 
greater than $1$ at $Q$ and $v_P(hj-sl)=0$ as $v_P(ad-bc)=0$.
As $y$ is holomorphic at $Q$, we have that $j(Q)\neq 0$ since
\begin{itemize}
\item if $s(Q)\neq 0$, then $j(Q)\neq 0$ since $y(Q)=s(Q)/j(Q)$
\item suppose that $s(Q)=0$, then $j(Q)$ cannot be zero as $h(Q)j(Q)-s(Q)l(Q)\neq 0$.
\end{itemize}
If we can show that $z$ has valuation $1$ we are done, as we would have by \eqref{equatonforz}
\[\cO_Q=\cO_P[z]=\cO_P[y].\]
Since $v_Q(j)=0$, we have
\[1\leq v_Q(z)=1-v_Q(l\tilde y+ s)\leq 1\]
which concludes the proof in the case in which $y$ is holomorphic at $Q$.
If $y$ is not holomorphic at $Q$, we choose $\tilde y$ as before and $t=1/y$. Now it holds $t=\frac{l \overline y+ j}{h\overline y + s}$ and the previous considerations can be applied again to show $\cO_P[t]=\cO_Q$.

%
%
 



\end{proof}

\subsection{On the density of totally ramified extensions of fixed degree}

The following lemma gives the exact restriction for the set of possible shifts which are candidates for turning a given poylnomial $f(X)$ into an Eisenstein polynomial. 

\begin{lemma}
Let $\cO_P$ be a valuation ring of a function field $F$ and let $f\in \cO_P[T]$ of degree $d=\deg(f)\geq 2$. Let $G_P\subseteq \cO_P$ be a set of representatives of
$\cO_P/P$.  
We have that $f(T+i)$ is Eisenstein with respect to $P$  for some $i\in \cO_P$ if and only if $f(T+i)$ is Eisenstein with respect to $P$ for some $i\in G_P$.
\end{lemma}
\begin{proof}
One implication is obvious so let us prove the other direction.
Suppose that $f(T)=\sum^d_{i=0} a_i T^i$ for $a_i\in \cO_P$ and that $f(T+i)$ is $P$-Eisenstein for some $i\in \cO_P$.
Let us first show that, for any $p\in P$, the polynomial $f(T+p)$ is Eisenstein if and only if $f(T)$ is Eisenstein. This follows by the fact that the conditions modulo $P$ of $f$ are easily verified
as $f(T+p)=f(T) \mod P$. The condition modulo $P^2$ reduces to check that $f(p)\neq 0 \mod P^2$, which indeed holds as $f(p)=\sum^d_{i=0} a_i p^i=a_0\neq 0\mod P$.
The proof is now straightforward, since any $i\in \cO_P$ can be written as $j+p$ for some $j\in G_P$ and $p\in P$: we have that  $f(T+i)=f(T+j+p)$ is $P$-Eisenstein and then $f(T+j)$ is $P$-Eisenstein for $j\in G_P$.
\end{proof}

We are now ready to extract the density of the set the set of polynomials in question using the local to global principle together with some tools from measure theory and linear algebra. With a small abuse of notation, in what follow we will identify with $\vH^{n+1}$ the set of degree $n$ polynomials over a holomorphy ring $\vH$.
\begin{theorem}\label{thm:probabilitytotramextension}
Let $\vH_{\cS}$ be a holomorphy ring of a global function field $F/\mathbb F_q$ having full constant field $\vF_q$. The set of polynomials $f\in \vH[T]$ of degree $n\geq  3$ such that the ring $E=F[y]/(f(y))$ is a field and  $E:F$ is nicely totally ramified with respect to $f$, has density 
\[R(n,\vH)=1-\prod_{P\in \cS} \frac{(q^{\deg(P)}-1)^2(q^{\deg(P)}+1)}{q^{\deg(P)(n+2)}}.\]
\end{theorem}
\begin{proof} 
First, we should observe that we can restrict to the separable case, as the set of inseparable polynomials has density zero. 
Let us denote  the complement of $\cS$ in $\vP_F$ by $\cS^c$.
First, we should prove that if ramification occurs, it occurs with probability one at a place of $\cS$. In fact, the polynomials $f(T) \in \vH_{\cS}[T]$ defining a totally ramified extension $E:F$ at a place of $\cS^c$, do not contribute to  $R(n,\vH)$. This is not surprising, as the definition of the density depends on the choice of $\cS$. In particular we prove:
\\ \textbf{Claim 1.} \emph{The density of degree $n$ polynomials in $\vH_{\cS}[T]$ for which $E:F$ is nicely totally ramified equals the density of degree $n$ polynomials in $\vH_{\cS}[T]$ for which $E:F$ is nicely totally ramified at a place of $\cS$.}\\
\emph{Proof of Claim 1.}\\
Let \[A:=\{f\in \vH_{\cS}[T]: \; f \; \mbox{ is Eisenstein for some shift $i\in G_P$ and some place}\; P\in \cS^c \}\]
and 
\[B=\{f\in \vH_{\cS}[T]: \quad T^nf(1/T) \; \mbox{ is Eisenstein for some place}\; P\in \cS^c \}\]  
It is enough to show that $C=A\cup B$ has density zero.
First, we show that if $f\in C$, then $f\in \cO_P[T]$ for some 
$P\in \cS^c$. 
Clearly, if $f\in B$ this follows immediately by the definition of Eisenstein polynomial. Let now suppose that 
$f\in A$. Let $f=\sum^n_{j=0}a_j T^j$ and $i\in G_P$ such that $f(T+i)=\sum^n_{j=0}b_j T^j$ is Eisenstein.
It is necessary now to observe that  the map defined by the shift $\sigma_i: F^{n+1}\longrightarrow F^{n+1}$ defined by $\sigma_i(f)=f(T+i)$ is linear, upper triangular in the basis $\{1,\dots, T^n\}$, and any element on the diagonal is $1$. This shows recursively that if $a_{j}\in \cO_P$ for all $j\in \{\ell+1, \dots, n\}$, then $a_{\ell}\in \cO_P$. This is enough to conclude, since $a_n=b_n\in \cO_P$.

Thanks to what we just proved, it is now clear that $C\subseteq \bigcup_{P\in \cS^c}\vH_{\cS\cup P}^{n+1}$. 
Since the complement of $\cS$ is finite, we have that 
\[\overline{\vD}_{\cS}(C)\leq \sum_{P\in \cS^c}\overline{\vD_{\cS}}(\vH_{\cS\cup P}^{n+1}).\]
It follows that if we can show 
$\overline{\vD}_{\cS}(\vH_{\cS\cup P}^{n+1})=0$ we are done.
To see this, let us consider for $D\in \cD_\cS$
the quantity $Q(D,P):=|\cL(D)\cap \vH_{\cS\cup P}|$. For a divisor
 $D\in \cD_\cS$ of large degree, $Q(D,P)$ is easy to estimate, as one can write $D=D'+mP$
with $P\notin \supp(D')$ and $m$ positive integer. Then
$Q(D,P)=Q(D',P)=  q^{\deg(D')+1-g}$.
It follows that 
\[\frac{|Q(D,P)|}{q^{\ell(D'+mP)}}=\frac{|Q(D',P)|}{q^{\ell(D'+mP)}}\leq C q^{-m\deg(P)}\]
which concludes the proof, as
\[\overline{\vD_{\cS}}(\vH_{\cS\cup P})=\limsup_{D\rightarrow \infty} \frac{ |Q(D,P)|^{n+1}}{q^{(n+1)\ell(D)}}.\]
\emph{End of the proof of Claim 1}\\
From now on in this proof $\cS$ will be fixed, therefore we will denote $\vH_\cS$ and $\vD_\cS$ by $\vH$ and $\vD$ respectively.
For any $P\in \cS$, let $\widehat \cO_P$ be the $P$-adic completion of $\cO_P$. With a small abuse of notation, let us identify  with $\widehat{\cO}_P^{n+1}$ the set of polynomials of degree $n$ in $\widehat{\cO}_P[T]$. Let now $G_P$ be a set of representatives of the elements $\widehat{\cO}_P/P$. For any $a \in G_P$ we define the linear maps
\[\sigma_a:\widehat{\cO}_P^{n+1}\longrightarrow\widehat{\cO}_P^{n+1}\]
\[p(T)\mapsto p(T+a)\]
and
\[\inv: \widehat{\cO}_P^{n+1} \longrightarrow \widehat{\cO}_P^{n+1}\]
\[p(T)\mapsto T^n p\left(\frac{1}{T}\right).\]
Let $V_P$ be the set of Eisenstein polynomial of degree $n$ over $\widehat\cO_P$ i.e. 
\[V_P=\underbrace{(P\widehat{\cO}_P\setminus P^2\widehat{\cO}_P)\times P\widehat{\cO}_P\times\cdots \times P\widehat{\cO}_P\times (\widehat{\cO}_P \setminus P\widehat{\cO}_P)}_{n+1}. \]
We  want now to compute the measure of the set of degree $n$ polyomials which are Eisenstein after some shift or inversion. \\
\textbf{Claim 2.} \emph{The  $P$-measure of $U_P:=\left(\bigcup_{a\in G_P} \sigma_a(V_P)\right)\,\cup \inv(V_P)$ is 
\[\mu_P(U_P)=q^{-\deg(P)(n-1)}(1-q^{\deg(P)})^2(q^{\deg(P)}+1).\]}\\
\emph{Proof of Claim 2.}\\
We first want to show that the union above is disjoint. 
First, let us show that for any $a\neq 0$ in $G_P$ we have $\sigma_a(V_P)\cap V_P=\emptyset$. This is easy, since for $g(T)\in \sigma_a(V_P)\cap V_P$, we have that the degree zero coefficient of $g_a(T)=g(T+a)$ is $g(a)$ and $g(T)\in V_P$ which implies that it can be written as 
$g(T)=b T^n+p\, h(T)$ for $p\in P$, $b\in \cO_P\setminus P$: it follows that $g(a)\notin P$.
As an easy consequence of the previous fact, we have that $\sigma_a(V_P)\cap \sigma_{a'}(V_P)=\emptyset$ for $a,a'\in G_P$ and $a\neq a'$. In addition, the intersection of $\inv(V_P)$ with any of the $\sigma_a(V_P)$ is also empty. To see this, suppose that $g_a(T)\in \inv(V_P)$. The fact that the degree $n$ coefficient of $g_a(T)$ is in $\cO_P\setminus P$, is in contradiction with the fact that $g_a(T)\in \inv(V_P)$.
In order to get the final claim, it is enough to compute the measures of each one of the $\sigma_a(V_P)$ and of $\inv(V_P)$.
It is easy to see that the maps $\sigma_a$ and $\inv$ are indeed $F$-linear maps with determinant one, therefore $\mu_P(\sigma_a(V_P))=\mu_P(\inv(V_P))=\mu_P(V_P)$ for any $a\in G_P$. The problem is then reduced to computing $\mu_P(V_P)$. 
As $V_P$ is diagonal, one can obtain the measures of each of its component and then compute the product in order to get the wanted result. These measures are indeed easy to compute
\[\mu_P(P\widehat{\cO}_P)=q^{-\deg(P)},\] \[\mu_P(P\widehat{\cO}_P\setminus P^2\widehat{\cO}_P)=q^{-\deg(P)}-q^{-2\deg(P)},\] and 
\[\mu_P(\widehat{\cO}_P \setminus P\widehat{\cO}_P)=1-q^{-\deg(P)}.\]
It follows that
\[\mu_P(U_P)=\mu_P(\inv(V_P))+\sum_{a\in G_P}\mu_P(V_P)=(1+q^{\deg(P)})\mu_P(V_P)\]
\emph{End of the proof of Claim 2}\\
Clearly, we have an embedding $\iota_P$ of $\vH$ into $\widehat\cO_P$. Using this embedding and Proposition \ref{fund_thm_ramified} we have that the polynomial $f\in \vH[T]$ defines a nicely totally ramified extension of $F$ at $P$ if and only if $\iota_P(f)\in U_P$. In fact, if $f$ defines a nicely totally ramified extension $E=F[y]/f(y)$ then we are in one of the two cases of Proposition \ref{fund_thm_ramified}, as either $\cO_P[y]=\cO_Q$ or $\cO_P[1/y]=\cO_Q$ for some $Q|P$ totally ramified, which implies that either $\iota_P(f)\in \sigma_a(V_P)$ for some $a\in G_P$ or $\iota_P(f)\in \inv(V_P)$. 
Vice versa, suppose that $\iota_P(f)\in U_P$, then $f(T+a)$ is Eisenstein for some $a$ or $T^nf(1/T)$ is Eisenstein. This implies that the extension $F[y]/(f(y))$ is totally ramified by Lemma \ref{propcombinationstich}, as  $F(y+a)=F(1/y)=F(y)$. 
We have now to show that the total ramification is nice. We observe that, if $Q|P$ is the totally ramified extension then
$\cO_Q$ is the integral closure of $\cO_P$ by Lemma \ref{propcombinationstich}. Therefore, if $\iota_P(f)\in \sigma_a(V_P)$ for some $a\in G_P$, then $y$ is holomorphic at $Q$, otherwise, if $\iota_P(f)\in \inv(V_P)$, then $1/y$ is holomorphic at $Q$, from which it follows that we end up again in one of the cases of Proposition \ref{fund_thm_ramified}.\\
\textbf{Claim 3.} \emph{Let $t$ be a positive integer and $\cS_t$ be the set of places of $\cS$ of degree larger than $t$. We have}
\[\lim_{t\rightarrow \infty }\overline{\vD}(\{a\in \vH^{n+1}\:|\: a\in U_P \;\text{for some }\; P\in \cS_t\})=0.\]
\emph{Proof of Claim 3.}\\
By recalling the definition of $U_P$ we can split the limit above as $\lim_{t\rightarrow \infty} \overline{\vD}(I_t)+\overline{\vD}(A_t)$ where
\[I_t=\{a\in \vH^{n+1}\:|\: a\in \inv(V_P) \;\text{for some }\; P\in \cS_t\}\]
and
\[A_t=\{a\in \vH^{n+1}\:|\: a\in \bigsqcup_{a\in G_P} \sigma_a(V_P) \;\text{for some }\; P\in \cS_t\}.\]
Let us first deal with the limit in $I_t$, which is the easy part.
By recalling the definition of the linear operator $\inv$, and by considering only the conditions on the coefficients of the terms of degree $n$ and $n-1$, it is easy to observe that 
\[I_t\subseteq \{(x_0,\dots, x_n)\in \cL(D)^n\;|\; x_{n-1}\equiv x_n\equiv 0 \mod P \;\text{for some }\; P\in \cS_t\}=I_t'.\]
As $x_n$ and $x_{n-1}$ are indeed coprime polynomials, $\lim_{t\rightarrow \infty} \overline{\vD}(I_t)=\lim_{t\rightarrow \infty} \overline{\vD}(I_t')=0$ by Theorem \ref{condition_verified_polynomials_THEOREM}.

It remains to show that $\lim_{t\rightarrow \infty} \overline{\vD}(A_t)=0$. 
We now produce two equations associated to $A_t$ via elementary elimination theory, the final claim will follow again by Theorem \ref{condition_verified_polynomials_THEOREM}.
Let us write a generic shift for a generic polynomial of degree $n$:
\[f(T)=\sum^n_{i=0} x_i T^i\in F[x_1,\dots, x_n,T]\]
\[f(T+z)=\sum^n_{i=0}c_i(x_i,\dots, x_n,z)T^j\in F[x_1,\dots, x_n,z,T],\]
for some polynomial functions $c_i(x_i,\dots, x_n,z)\in F[x_i,\dots, x_n,z]$. Notice that the polynomial $c_i$ depends only on the last $n-i$ variables since the shift map is upper triangular.
In particular, let $(y_0,\dots, y_n)$ be a specialization to elements of $\vH$ such that the given polynomial $f(T)=\sum^n_{i=0} y_i T^i$ is in $A_t$. Then,  there exists $\overline z\in \vH$ such that $c_0(y_0,\dots y_n,\overline z)\equiv c_1(y_1,\dots,y_n,\overline z)\equiv c_{n-1}(y_{n-1},y_n,\overline z)\equiv 0 \mod P$ for some $P\in \cS$.
Let us now explicitly compute $c_0,c_1$ and $c_{n-1}$:
\[c_0(x_0,\dots x_n, z)=f(z)=\sum^n_{i=0} x_iz^{i}\]
\[c_1(x_1,\dots x_n, z)=f'(z)=\sum^n_{i=1} i x_iz^{i-1}\]
\[c_{n-1}(x_{n-1},x_n, z)=x_{n-1}-\tilde{n}x_n z,\]
where we set $\tilde n=-n$ to simplify the computations that follow. Observe that $c_{n-1}$ is distinct from $c_1$ as $n>2$.
Recall that we want to apply Theorem \ref{condition_verified_polynomials_THEOREM} so we have to eliminate the variable $z$, getting two equations. In order to do that, we have to distinguish two possibilities, depending on whether the characteristic of $F$ divides $n$ or not.
In case $n$ is divisible by $\Char(F)$, then two polynomials which eliminate $z$ are $\res_z(f(z),f'(z))=\disc_z(f(z))\in F[x_0\dots,x_n]$ and $x_{n-1}$.
Define $A_t'=\{f\in \vH^{n+1}\:|\: \disc_z(f)\equiv x_{n-1}\equiv 0\mod P \;\text{for some }\; P\in \cS_t\}$. By Theorem \ref{condition_verified_polynomials} we get 
that $\lim_{t \rightarrow \infty} \overline{\vD}(A_t')=0$.
Observe that by construction $A_t\subseteq A_t'$, from which the claim follows.

It remains to deal with the case in which $n$ is not divisible by $\Char(F)$.
We can eliminate the variable $z$ from the ideal 
\[\langle c_0(x_0,\dots x_n, z),c_1(x_1,\dots x_n, z),c_{n-1}(x_{n-1},x_n, z)\rangle\subseteq F[x_0,\dots, x_n, z]\]
by multiplying $c_0$  by $(\tilde n x_n)^{n-1}$ and $c_1$ by $(\tilde n x_n)^{n-2}$ and using the relation given by $c_{n-1}$, getting two polynomials: 

\[f_0(x_0,\dots,x_n):=x^n_{n-1}\tilde n^{-1}+\sum^{n-1}_{i=0} (\tilde n x_n)^{n-1-i} x_i x_{n-1}^i\]
\[f_1(x_1,\dots,x_n):=-x^{n-1}_{n-1}+\sum^{n-1}_{i=1} i(\tilde n x_n)^{n-2-(i-1)} x_i x_{n-1}^{i-1}.\]

We define now $A_t'\supseteq A_t$ as the set of specializations $(y_0,\dots,y_n)\in \vH^{n+1}$ such that
$f_0(y_0,\dots,y_n)\equiv f_1(y_1,\dots,y_n)\equiv 0 \mod P$ for some place $P$ of degree larger than $t$.
It remains to show that $f_0$ and $f_1$ are coprime, then we will get $\lim_{t\rightarrow \infty}\overline{\vD}(A_t)\leq \lim_{t\rightarrow \infty}\overline{\vD}(A_t')=0$ by Theorem \ref{condition_verified_polynomials}. First, observe that $f_1\neq 0$ as $c_1\neq c_{n-1}$ (since $n>2$).
Notice that $f_0=x_0 h_1(x_1,\dots,x_n) \tilde n^{n-1}+h_0(x_1,\dots,x_n)$ for some $h_0, h_1\in \vH[x_1,\dots, x_n]$. On the other hand, $x_0$ does not appear in $f_1$, which implies that, if there is an irreducible common factor $g$ of $f_0$ and $f_1$, then
it must divide both $h_0$ and $h_1$. In our specific case case, $h_1(x_1,\dots,x_n)=\tilde n^{n-1}x_n^{n-1}$ which forces $g=x_n$ but this is impossible, as 
for example $f_1(x_0,\dots,x_{n-1},0)\neq 0$. 

\emph{End of the proof of Claim 3}\\
We are now able to conclude the proof using Theorem \ref{thm:main_density_function_field}.
Let $R$ be the set of degree $n$ polynomials giving rise to a nice totally ramified extension and let us consider $R^c$.
By looking at the definition of the map $\pi$ in the Theorem \ref{thm:main_density_function_field} and by recalling the choice of the $U_P$'s, we have that 
$R^c=\pi^{-1}(\{\emptyset\})$.
Therefore, since Condition \eqref{fund_cond_loc_to_glob} of Theorem \ref{thm:main_density_function_field} is verified by the previous claim, we have that
\[\vD(R^c)=\vD(\pi^{-1}(\{\emptyset\}))=\nu(\{T\})=\prod_{P\in \cS\setminus T } (1-\mu_P(U_P))\]
from which the claim follow by Claim 2 and the fact that $\vD(R)=1-\vD(R^c)$.

\end{proof}

In the proof of the previous result the restriction $n>2$ is crucial only for the proof of Claim 3. In the following remark we briefly adapt the above strategy to the case $n=2$. 

\begin{remark}
For the sake of completeness we should also deal with the case $n=2$.
Let $R$ be the set of separable irreducible degree $2$ polynomials $f\in \vH[T]$ for which the extension $F[y]/(f(y))$ is nicely totally ramified. 
For this, let $\cS_t^c$ be the set of places of $\cS$ of degree at most $t$. For $P\in \cS_t^c$, let $U_P$ be defined as in the proof of Theorem \ref{thm:probabilitytotramextension}  and for $P\in \cS_t$ let $U_P=\emptyset$. For any fixed $t$, one easily observes that 
\[E_t:=\pi^{-1}(\{\emptyset\})\supseteq R^{c}\]
therefore 
\begin{equation}\label{eq:deg2}
\lim_{t\rightarrow \infty } \vD(E_t)\geq \vD(R^{c}).
\end{equation}
Now, for any fixed $t\in \vN$, Condition \eqref{fund_cond_loc_to_glob} is verified, since the set of $U_P$ is finite.
Using Theorem \ref{thm:main_density_function_field} as before one gets that the density of $E_t$ is
\[\vD(E_t)=\prod_{P\in \cS_t^c} \frac{(q^{\deg(P)}-1)^2(q^{\deg(P)}+1)}{q^{\deg(P)(4)}}\]
which is now a finite product.
Letting $t$ go to plus infinity, the sequence $\vD(E_t)$ converges to zero, which forces $\vD(R^c)=0$ and so 
$\vD(R)=1$.
\end{remark}
\begin{remark}
The reader should notice that the condition on the fact that the total ramification is good is of fundamental importance in the proof of the result: in fact,  it allows the conversion given by Proposition \ref{fund_thm_ramified}. It would be of great interest to understand what happens to the final density result of Theorem \ref{thm:probabilitytotramextension} when the condition in Definition \ref{def:goodtotram} is removed.
\end{remark}

\section{The density of rectangular Unimodular Matrices}\label{sec:unimodularmatrices}

Using the local to global principle, in this section we close the problem of computing the density of rectangular unimodular matrices with entries in an integrally closed subring of a global function field. As a special case of the main result, we also get the density of such matrices with entries the ring $\vF_q[x]$. Such matrices are of deep interest in coding theory since they are one of the key ingredients to define a convolutional code. In that context a rectangular unimodular matrix is usually said to be \emph{left prime}, see for example \cite[pg 122]{FORNASINI2004119} for a characterization of such condition.
\begin{definition}
Let $k,m$ be positive integers such that $k<m$. Let $M$ be a $k\times m$ matrix having entries in a commutative ring $R$.
We say that $M$ is \emph{rectangular unimodular} if it can be extended with $m-k$ rows in $R^n$ to an element of $GL_m(R)$. 
\end{definition}
The reader should notice that in the definition above we require the extension of $M$ to be an automorphism of $R^m$ as an $R$-module. 

\begin{problem}
What is the density of rectangular unimodular matrices of $\vH^{k\times m}$?
\end{problem}
The problem was first addressed in \cite{ guo2013probability} in the case of $\vH=\vF_q[x]$.  Unfortunately, the proof in \cite{guo2013probability} is not correct  as a priori there is no reason for the infinite sum in the proof of \cite[Theorem 1]{guo2013probability} to commute with the superior limit. This was also observed in \cite{micheli2016density}, where it is present a partial fix of the issue, in the case in which the number of columns are at least the double of the number of rows.
In what follows we provide the complete solution of the problem using the tools we developed in the first section for the general case of holomorphy rings. In order to do this, we have to construct the $p$-adic system of subsets used in Theorem \ref{thm:main_density_function_field}. This can be done using the following result proved in \cite{gustafson1981matrix}.

\begin{theorem}\label{dedekindextension}
Let $R$ be a Dedekind domain. Let $M$ be a $k\times m$ matrix with entries in $R$ and $I_M$ be the ideal generated by the determinants of the maximal minors of $M$. 
Then $M$ is rectangular unimodular if and only if $I_M=R$.
\end{theorem}

We are now able to state and prove the main theorem

\begin{theorem}\label{unimatrixdensity}
Let $F/\mathbb F_q$ be a function field with full constant field $\mathbb \vF_q$.
Let $\cS$ be a set of places in $\mathbb P_F$ having finite complement.
The density of of the set $U$ of $k\times m$ rectangular unimodular matrices with entries in $\vH_\cS$ is
\[\vD_\cS(U)=\prod^m_{i=m-k+1}\frac{1}{\zeta_\vH(i)},\]
where $\zeta_\vH$ is the zeta function of $\vH$ defined as 
\[\zeta_\vH(s)=\zeta_F(s)\cdot\prod_{R\in \vP_F\setminus \cS}(1-q^{-\deg(R)s})=\prod_{P\in \cS}(1-q^{-\deg(P)s})^{-1}\]
and $\cS$ is the holomorphy set of $\vH$.
\end{theorem}
\begin{proof}
For any $P\in \cS$, let $U_P$ be the set of non-unimodular matrices in $\widehat \cO_P^{k\times m}$, i.e. the set of matrices for which the ideal generated by the determinant of the maximal minors is contained in $P \widehat \cO_P$. By Theorem \ref{dedekindextension} we have that a matrix $M$ is in $U$ if and only if $M\notin U_P$ for any $P$.
Therefore, in the notation of Theorem \ref{thm:main_density_function_field} it follows that $\pi^{-1}(\{\emptyset\})=U$ and then, if the system 
$\{U_P\}$ verifies condition \eqref{fund_cond_loc_to_glob} we have that
\[\vD_\cS(U)=\prod_{P\in \cS} (1-\mu_P(U_P)).\]

In the polynomial ring 
\[
\vH[x_{1,1}, x_{1,2},\dots,x_{1,m}, x_{2,1},\dots x_{i,j}, \dots x_{k,m}].\]
let us consider the $k\times m$ matrix whose $i,j$ entry is the variable $x_{i,j}$. Let $N$, $N'$ be two distinct $k\times k$ minors of this matrix and $f$ and $g$ be the determinants of $N$ and $N'$ respectively. Notice that $f$ and $g$ are distinct irreducible polynomials, therefore coprime.
It is straightforward  observe that the set $\{a\in \vH_{\cS}^d\:|\: a\in U_P \;\text{for some }\; P\in \cS_t\}$ is contained in $\{a\in \vH_{\cS}^d\:|\: f(a)\equiv g(a) \equiv 0 \mod P \;\text{for some }\; P\in \cS_t\}$. By applying Theorem \ref{condition_verified_polynomials} one gets that the condition \ref{fund_cond_loc_to_glob} is verified.

It remains to compute the $P$-measure of the $U_P$, the claim will follow by Proposition \ref{fund_thm_ramified} and the fact that $\pi^{-1}(\{\emptyset\})=U$.
In order to compute such measures, we first decompose $U_P$.
Let $r$ be a representative in $\widehat \cO_P^{k\times m}$ for a matrix in
$(\cO_P/P)^{k\times m}$ which is not full rank. We observe that the set $V_{P}(r)=\{M\in \cO_P^{k\times m}: M=r+(P\widehat \cO_P)^{k\times m}\}$ is contained in $U_P$. 
In addition, if $r,r'\in \cO_P^{k\times m}$ and satisfy $r\not\equiv r' \mod P$, then $V_{P}(r)\cap V_{P}(r')=\emptyset$. Let $R_P\subseteq \widehat \cO_P$ be a set of representatives for the non-full rank matrices of $(\cO_P/P)^{k\times m}$.
Therefore, for fixed $P$ the union of such $V_{P}(r)$ is disjoint and covers the whole $U_P$, from which it follows that
\[\mu_P(U_P)=\mu_P(\bigsqcup_{r\in R_P} V_{P}(r))=\sum_{r\in R_P} \mu_P(V_{P}(r)).\]
Since
\[\mu_P(V_P(r))=\mu_P(V_{P}(r)-r)=\mu_P(P\widehat{\cO}_P)^{mk}=q^{-mk\deg(P)}\]
and the number of non-full rank matrices in $(\cO_P/P)^{k\times m}$ is \[q^{mk\deg(P)}-\prod^{k-1}_{i=0}\left(q^{\deg(P)m}-q^{\deg(P)i}\right),\]
one gets that
\[\mu_P(U_P)=1-q^{-mk\deg(P)}\prod^{k-1}_{i=0}\left(q^{\deg(P)m}-q^{\deg(P)i}\right)=1-\prod^{k-1}_{i=0}\left(1-q^{-\deg(P)(m-i)}\right).\]
Therefore, by applying Theorem \ref{thm:main_density_function_field} we get $\vD_\cS(U)=\prod_{P\in \cS} (1-\mu_P(U_P))$. Exchanging now the finite product with the product over all the places, we get
\[\vD_\cS(U)=\prod^{k-1}_{i=0}\frac{1}{\zeta_\vH(m-i)}\]
which is the result we wanted.
\end{proof}
\begin{remark}
Notice that the density can be explicitly computed by the closed expression of the zeta function \cite[Corollary 5.1.12, (b)]{bib:stichtenoth2009algebraic} and the fact that $\vP_F\setminus \cS$ is finite. Observe that this is consistent with the result obtained for unimodular rows in the general case of global fields \cite{bib:HolMS,ferraguti2016mertens}.
\end{remark}
As an immediate corollary we have
\begin{corollary}
Let $k<m$ be positive integers. The density of the set $U$ of $k\times m$ rectangular unimodular matrices (i.e. left prime matrices) with coefficients in $\vF_q[x]$ is  
\[\vD(U)=\prod^m_{i=m-k+1}\frac{q^{i-1}-1}{q^{i-1}}.\]
\end{corollary}
\begin{proof}
Let $\vF_q(x)$ be the rational function field and $P_\infty$ be the place at infinity of $\vF_q(x)$. Recall that the zeta function of the rational function field is 
\[\zeta_{\vF_q(x)}(s)=\frac{1}{(1-q^{-s})(1-q^{-s+1})}.\]
Previously we already observed that 
\[\vF_q[x]=\bigcap_{P\neq P_\infty} \cO_P\]
 and therefore it follows that $\zeta_{\vF_q[x]}(s)=1/(1-q^{-s+1})$. The claim follows by applying directly Theorem \ref{unimatrixdensity}.

\end{proof}

\begin{remark}
An analogous result in the case of $\mathbb Z$ can be found in \cite{bib:maze2011natural}. It would be of interest to use Theorem \ref{thm:main_density_function_field} to compute the density of matrices having fixed Smith Normal form, at least in the case in which the holomorphy ring is a PID  (a result over $\mathbb Z$ is already available in \cite{wang2015smith}).
\end{remark}

\section*{Acknowledgements}
We would like to thank Bjorn Poonen for very interesting discussions and for sharing his deep insight of the problem. 
We thank Elisa Gorla for suggesting a sophisticated modification of the proof of Theorem \ref{condition_verified_polynomials_THEOREM}, which improved the paper in terms of length and clarity.
We would like to also thank Federico Amadio, Francesca Balestrieri and Andrea Ferraguti for reading the preliminary version of this manuscript.
This work started at MIT in Fall 2015 and was accomplished at the University of Oxford in Michaelmas 2016 under the Swiss National Science Foundation grant number 161757. 

\bibliographystyle{unsrt}
\bibliography{biblio}{}

\end{document}